\documentclass[12pt]{amsart}
\usepackage{amssymb}
\usepackage{amsfonts,amscd}
\usepackage{amsmath}
\usepackage[backref=page]{hyperref}
 \numberwithin {equation}{section}
\newtheorem{theo}{Theorem}[section]

\newtheorem{prop}[theo]{Proposition}
\theoremstyle{definition}
\newtheorem{remark}[theo]{Remark}

\newtheorem{defi}[theo]{Definition}
\newcommand{\bequ}{\begin{equation}} %
\newcommand{\eequ}{\end{equation}}
\newcommand{\bequs}{\begin{equation*}}
\newcommand{\eequs}{\end{equation*}}
\newcommand{\beqs}{\begin{equation*}}
\newcommand{\eeqs}{\end{equation*}}
\newcommand{\Real}{\mathbb R}
\newcommand{\Nat}{\mathbb N}

\newcommand{\epsic}{\varepsilon}
\newcommand{\eps}{\varepsilon}

\newcommand{\lbda}{\lambda}

\newcommand{\lra}{\longrightarrow}

\newcommand{\Ra}{\;\Rightarrow\;}
\newcommand{\sms}{\smallsetminus}

\newcommand{\sse}{\subseteq}
\newcommand{\mbx}{\mbox}
\newcommand{\ov}{\overline}

\newcommand{\ety}{\emptyset}
\newcommand{\dom}{\operatorname{dom}}
\newcommand{\norm}{\|\cdot\|}
\newcommand{\ben}{\begin{enumerate}}
\newcommand{\een}{\end{enumerate}}
\newcommand{\eite}{\end{itemize}}
\newcommand{\bite}{\begin{itemize}}

\newcommand{\rond}{\mathcal}

\begin{document}
\title[The strong Ekeland principle]{The strong Ekeland variational principle in  quasi-pseudometric spaces}
\author{S. Cobza\c{s} }
\address{\it Babe\c s-Bolyai University, Faculty of Mathematics
and Computer Science, 400 084 Cluj-Napoca, Romania}\;\;
\email{scobzas@math.ubbcluj.ro}
\begin{abstract}  Roughly speaking, Ekeland's Variational Principle (EkVP) (J. Math. Anal. Appl.   47  (1974), 324--353) asserts the existence of  strict minima of some  perturbed versions    of lower semicontinuous functions  defined on a complete metric space.  Later,    Pando Georgiev (J. Math. Anal. Appl. \textbf{131} (1988), no.~1,
  1--21) and Tomonari Suzuki  (J. Math.
  Anal. Appl. \textbf{320} (2006), no.~2, 787--794  and  Nonlinear Anal. \textbf{72} (2010), no.~5,
  2204--2209)),  proved a  Strong Ekeland Variational Principle, meaning  the existence of  strong minima   for such perturbations.  Note that Suzuki also considered
  the case of functions defined on Banach spaces, emphasizing the key-role played by reflexivity.

  In the last years an increasing interest was manifested by many researchers to extend EkVP to the asymmetric case, that is,  to quasi-metric   spaces (see the  references). Applications to optimization, behavioral sciences, and others, were obtained.   The aim of the present paper is to extend  the strong Ekeland  principle, both Georgiev and Suzuki versions, to the  quasi-pseudometric case.  At the end we ask for the possibility to extend it to asymmetric normed spaces (i.e., the extension of Suzuki's results).

\textbf{Classification MSC 2020:} 46N10 47H10   54E35 54E50 58E30\medskip


\end{abstract}

\date{\today; published in \textbf{Mathematics} \textbf{2024}}
\maketitle

\emph{Motto:}  "\emph{Je suis tr\`es honor\'e que le CEREMADE m'ait demand\'e de parler du principe dont je porte le nom}."  Ivar Ekeland, Paris 2018.
\footnote{``I am very honored that CEREMADE invited me to speak about the principle whose name I bear."\\
CEREMADE -  Centre de Recherche en Math\'ematiques de la D\'ecision, Paris}

\section{Introduction}

A variational principle is a proposition asserting that some
function, usually bounded  below and lower semi-continuous (lsc),  attains its minimum. If the original function does not
attain its minimum then one looks for an appropriate perturbation
such that the perturbed function has a minimum. Variational
principles have numerous applications to problems of optimization,
in the study of the differentiability properties of mappings, in
fixed point theory, etc. Their origins go back to the early stage of
development of the calculus of variations and are related to the
principle of least action from physics.

Ivar Ekeland announced  in 1972, \cite{ekl72}  (the proof appeared  in 1974 in \cite{ekl74}) a theorem asserting the existence of the minimum of a small perturbation of a lower semicontinuous (lsc) function defined on a complete metric space. This result, known as Ekeland Variational Principle (EkVP), proved to be a very versatile tool in various areas of mathematics and in applications  - optimization theory, geometry of Banach spaces, optimal control theory, economics, social sciences, and others. Some of these applications are presented by Ekeland himself in \cite{ekl79}.

At the same time, it turned out that this principle is equivalent to a lot  of results in fixed point theory (Caristi fixed point theorem), geometry of Banach spaces  (drop property), and  others (see  \cite{penot86}, for instance).

Since then, many extensions of this principle have been published, a good record being given in the book by Meghea \cite{Meghea}.

A version of EkVP in $T_1$-quasi-metric spaces was proved in \cite{cobz11}. The result was extended to arbitrary quasi-metric spaces in \cite{karap-romag15}, where it was shown that the validity of this principle actually
characterizes the right $K$-completeness of the underlying quasi-metric space. A fairly complete  presentation of various situations when the validity of a variational principle, of a maximality (minimality) principle, or of a fixed point result implies a kind of completeness of the underlying space (metric, generalized metric, ordered metric, or simply  just
ordered) is given in the paper \cite{cobz18}.

Other asymmetric versions (meaning quasi-metric, quasi-uniform or in  normed or locally convex asymmetric spaces) were proved in  \cite{kassay19}, \cite{bcs18}, \cite{cobz12}, \cite{cobz19}, \cite{cobz23a},  \cite{cobz23b}, and others.  Applications of these asymmetric versions to optimization theory, behavioral sciences and other social sciences were given in the papers  \cite{ak-soub23},  \cite{bcs18}, \cite{bm-soub15} and in other papers written by Antoine Soubeyran in cooperation with various mathematicians (consult MathSciNet, ZBMATH, Google.scholar  and other similar sites).

Strong versions of EkVP were proved by Georgiev \cite{pando86,pando88} and Suzuki  \cite{suzuki06,suzuki10}. The aim of this paper is to prove a quasi-metric version of the strong Ekeland Variational principle (see Section  \ref{S.str-Ek}).  Unfortunately, I was not able to extend Suzuki's results on the equivalence between  the strong EkVP and the reflexivity of the involved Banach space, nor the converse results, meaning completeness results implied by the validity of this principle,  formulated as open problems at the end of the paper.

\section{Ekeland and the strong Ekeland variational principles in metric  and Banach spaces}\label{S.str-Ek}
\subsection{Ekeland principle}

Ekeland \cite{ekl74} proved the following result, known as Ekeland Variational Principle (EkVP).

\begin{theo}[Ekeland Variational Principle]\label{t.EkVP}
Let $(X,d)$ be a complete metric space and $f:X\to
\Real\cup\{+\infty\}$ a lsc bounded below function. Let $\epsic
> 0$ and $ x_0\in \dom f.$

Then given $\lambda > 0$ there exists $\,z=z_{\epsic,\lambda}\in
X\,$ such that
\bequ \label{eq2.EkVP}
\begin{aligned}
{\rm (a)}&\; \quad f(z) +\frac{\epsic}{\lambda} d(z,x_0)
\leq f(x_0);\\
{\rm (b)}& \;\quad  f(z) < f(x) +\frac{\epsic}{\lambda} d(z,x) \quad\mbx{for all }\; x\in X\sms \{z\}.
\end{aligned}
\eequ

If, further,
$
 f(x_0) \leq \inf f(X) + \epsic,
$
then
$$
{\rm (c)}\; \quad  d(z,x_0)\leq \lambda.\qquad\qquad\qquad\qquad\qquad\qquad$$
\end{theo}

The Ekeland Variational Principle is sometimes written in the following form (see, for instance, \cite{penot86} or  \cite[Lemma 3.13]{Phe93}).

 \begin{theo}\label{t.EkVPb}
Let $(X,d)$ be a complete metric space and $f:X\to
\Real\cup\{+\infty\}$ a lsc bounded below function. Let $\epsic
> 0$ and $ x_0\in \dom f.$

Then given $\lambda' > 0$ there exists $\,z=z_{\lambda'}\in
X\,$ such that
\bequ \label{eq2.EkVPb}
\begin{aligned}
{\rm (a')}&\; \quad f(z) + \lambda' d(z,x_0) \leq f(x_0);\\
{\rm (b')}& \;\quad    f(z) < f(x) + \lambda' d(z,x)  \quad\mbx{for all }\; x\in X\sms \{z\}.
\end{aligned}
\eequ
If, further,
$
f(x_0) \leq \inf f(X) + \epsic,
$
then
$$
{\rm (c')}\; \quad   d(z,x_0)\leq \epsic/\lambda'\,.\qquad\qquad\qquad\qquad\qquad\qquad$$
\end{theo}

The equivalence of Theorems \ref{t.EkVP} and    \ref{t.EkVPb} follows by the substitution
\bequ\label{EkVP-a-b}
\lambda'=\frac\epsic\lambda \iff \lambda=\frac\epsic{\lambda'}\eequ.

  \subsection{The strong Ekeland variational principle}

  Let $X$ be a Banach space and $f:X\to\Real\cup\{\infty\}$ a function. A point $x_0\in \dom f$ is called

 \textbullet \;\;a \emph{minimum point} for $f$ if $f(x_0)\le f(x)$ for all $x\in X$;

\textbullet \;\;a \emph{strict minimum point} for $f$ if $f(x_0)<f(x)$ for all $x\in X\setminus\{x_0\}$;

\textbullet\;\;a \emph{strong minimum point} for if $f(x_0)=\inf f(X)$ and every sequence $(x_n)$ in $X$ such that
$\lim_nf(x_n)=\inf_X f$ is norm-convergent to $x_0$.

A sequence $(x_n)$ satisfying $\lim_nf(x_n)=\inf f(X)$ is called a \emph{minimizing sequence} for $f$.

\begin{remark}\label{re.str-min}
  A strong minimum point is a strict minimum point, but the converse is not true.
 \end{remark}

 Indeed, if there exist $z\ne z'$ such that $f(z)=m=f(z')$, where $m=\inf f(X),$ then the sequence $x_{2k-1}=z,\, x_{2k}=z', k\in \Nat,$ satisfies $\lim_nf(x_n)=m,$ but it is not convergent. Also,  the function $f:\Real \to\Real,\, f(x)=x^2e^{-x},$ has a strict minimum at 0, $f(0)=0,\, f(n)\to 0,$ but the sequence $(n)_{n\in\Nat}$ does not converge to 0.

 Condition (b$'$) in Theorem \ref{t.EkVPb}   asserts that, in fact, $z$ is  strict minimum   point for the perturbed function $\tilde f:=f+\lbda' d(z,\cdot)$. Georgiev \cite{pando86,pando88} proved a stronger variant of Ekeland variational principle, guaranteeing the existence of a strong minimum point $z$ for $\tilde f$.

 \begin{theo}[Strong Ekeland Variational Principle] \label{t.str-EkVP1} Let $(X,d)$ be a complete metric space and $f:X\to\Real\cup\{+\infty\}$ a lsc function bounded from below on $X$. Then for every $\gamma,\delta > 0$ and $x_0\in \dom f$ there exists $z\in X$ such that
\bequ\label{eq1.str-EkVP1}\begin{aligned}
 & {\rm (a)}\;\; f(z)+\gamma d(x_0,z)<f(x_0)+\delta;\\
  &{\rm (b)}\;\; f(z)<f(x)+\gamma d(z,x)\quad\mbox{for all}\quad x\in X\setminus\{z\}\,;\\
  &{\rm (c)}\;\;  f(x_n)+\gamma d(z,x_n)\to f(z)\;\Ra\; x_n\to z,\;\;\mbox{for every  sequence } (x_n)\mbx{  in }  X.
  \end{aligned} \eequ\end{theo}

  Georgiev, \emph{loc. cit.}, also showed the equivalence of this strong form of EkVP with  stronger forms of Danes' drop theorem, flower petal theorem, Phelps lemma, and others, extending so the results obtained by Penot \cite{penot86}. He gave a direct proof to the  strong drop theorem, the strong EkVP being a consequence of the equivalence mentioned above.  Later Turinici \cite{turin05} has shown that this strong form of EkVP can be deduced from Theorem \ref{t.EkVPb}.
  Let us mention that another proof was given by Deville, Godefroy and Zizler \cite{dev-god-ziz93} as a consequence of their generic smooth variational principle.

  Observe that there is a discrepancy between the conditions (a$'$) in Theorem \ref{t.EkVPb} and condition (a) in Theorem \ref{t.str-EkVP1}, condition (a$'$) being stronger than (a). As was remarked by Suzuki \cite{suzuki06,suzuki10}, a strong version of the Ekeland variational principle with condition (a$'$)
instead of (a) can be proved by imposing supplementary conditions on the underlying metric (or Banach) space $X$, which are, in some sense, also necessary.

Let  $f\colon X\to(-\infty,+\infty]$ be  a proper function defined on a metric space $(X,\rho)$. For $x_0\in \dom f$ and $\lambda>0$ consider an element  $z=z_{x_0,\lambda} $ satisfying  the following conditions:
\bequ\label{eq1.str-EkVP-Suz}\begin{aligned}
  &{\rm (i)}\;\; f(z)+\lbda \rho(z,x_0)\le f(x_0)\,;\\
  &{\rm (ii)}\;\; f(z)<f(x)+\lbda \rho(z,x)\quad\mbox{for all}\quad x\in X\setminus\{z\}\,;\\
  &{\rm (iii)}\;\;  f(x_n)+\lbda \rho(z,x_n)\to f(z)\;\Ra\; x_n\to z,\;\;\mbox{for every  sequence } (x_n)\mbx{  in }  X.
  \end{aligned}\eequ

  If $(X,\norm)$ is a normed space, then $\rho(x,y)$ is replaced by $\|y-x\|.$

  A metric space $(X,\rho)$ is called \emph{boundedly compact} if every bounded closed subset of $X$ is compact, or equivalently, if every bounded sequence in $X$ contains a convergent subsequence.
  \begin{remark}\label{re.bd-comp}
    It is obvious that a boundedly compact metric space is complete, and that a normed space is boundedly compact if and only if it is finite dimensional.
  \end{remark}

  \begin{theo}[\cite{suzuki06}]\label{T1.str-EkVP-comp} Let $(X,\rho)$ be a boundedly compact metric space, $f\colon X\to(-\infty,+\infty]$ a lsc bounded from below function,
  $x_0\in \dom f$ and $\lbda >0.$

  Then there exists a point $z\in X$ satisfying the conditions  \eqref{eq1.str-EkVP-Suz}.
   \end{theo}

   \begin{remark}\label{re.q-cv} 1.  Let $X$ be a vector space. A function $f:X\to\Real\cup\{\infty\}$ is called \emph{quasi-convex} if
 $$
 f((1-t)x+ty)\le \max\{f(x),f(y)\}\,,$$
 for all $x,y\in X$ and $t\in[0,1].$ This is equivalent to the fact that the sublevel  sets $\{x\in X:f(x)\le\alpha\}$ are convex for all $\alpha\in\Real$  (see \cite{Niculescu}).

 2. One says that a Banach space $X$ is a \emph{dual Banach space} if there exists a Banach space $Y$ such that $Y^*=X.$  Obviously, a reflexive Banach space is a dual Banach space with $X=\left(X^*\right)^*$  and, in this case, the weak  (i.e. $\sigma(X^*,X^{**})$) and the weak$^*$ (i.e. $\sigma(X^*,X)$) topologies on $X$  agree.
 \end{remark}

In the Banach space case   the following results can be proved.

\begin{theo}[\cite{suzuki06}]\label{T1.str-EkVP-refl} Let $X$ be a   Banach space,   $f\colon X\to(-\infty,+\infty]$ a bounded from below function, $x_0\in \dom f$ and $\lbda >0.$
\ben
\item[\rm 1.] If $X$ is a dual Banach space and $f$ is  $w^*$-lsc,
  then there exists a point $z\in X$ satisfying   \eqref{eq1.str-EkVP-Suz}  with $x_n\xrightarrow{w^*}x$ in the condition (iii).
\item[\rm 2.]
Suppose that  the  Banach space $X$ is  reflexive. If  $f$ is  weakly  lsc,  then there exists a point $z\in X$ satisfying the conditions  \eqref{eq1.str-EkVP-Suz}.
The same is true if $f$ is quasi-convex and norm-lsc.
   \een\end{theo}

    As it was shown by Suzuki \cite{suzuki10}, in some sense, the results from Theorems \ref{T1.str-EkVP-comp} and \ref{T1.str-EkVP-refl}  are the best that can be expected.

 \begin{theo}\label{t.str-EkVP-comp}
 For a   metric space $(X,\rho)$ the following are equivalent.
 \begin{enumerate}
   \item[{\rm 1.}] The metric space $X$ is boundedly compact.
   \item[{\rm 2.}] For every  proper  lsc bounded from below function $f\colon X\to(-\infty,+\infty]$,     $x_0\in \dom f$ and $\lbda >0$
there exists a point $z\in X$ satisfying the conditions   \eqref{eq1.str-EkVP-Suz}.
 \item[{\rm 3.}] For every  Lipschitz function $f\colon X\to[0,+\infty)$, \,    $x_0\in \dom f$ and $\lbda >0$
there exists a point $z\in X$ satisfying the conditions  \eqref{eq1.str-EkVP-Suz}.
 \end{enumerate}
 \end{theo}

 A similar result holds in the case of normed spaces.

 \begin{theo}\label{t.str-EkVP-refl}
 For a   normed  space $(X,\norm)$ the following are equivalent.
 \begin{enumerate}
   \item[{\rm 1.}]  $X$ is a reflexive Banach space.
   \item[{\rm 2.}] For every  proper  lsc bounded from below quasi-convex function $f\colon X\to(-\infty,+\infty]$,     $x_0\in \dom f$ and $\lbda >0$
there exists a point $z\in X$ satisfying the conditions   \eqref{eq1.str-EkVP-Suz}.
 \item[{\rm 3.}] For every  Lipschitz convex function $f\colon X\to[0,+\infty)$, \,    $x_0\in \dom f$ and $\lbda >0$
there exists a point $z\in X$ satisfying the conditions     \eqref{eq1.str-EkVP-Suz}.
 \end{enumerate}
 \end{theo}

   \section{The case of quasi-pseudometric  spaces}

   We present in this section some versions of Ekeland and strong Ekeland principles in quasi-pseudometric  spaces.
   \subsection{Quasi-pseudometric spaces}\label{Ss.qpm}

    A {\it quasi-pseudometric} on an arbitrary set $X$ is a mapping $d: X\times X\to
[0,\infty)$ satisfying the following conditions:
\begin{align*}%
\mbox{(QM1)}&\qquad d(x,y)\geq 0, \quad and  \quad d(x,x)=0;\\
\mbox{(QM2)}&\qquad d(x,z)\leq d(x,y)+d(y,z),  %
\end{align*}%
for all $x,y,z\in X.$ If further
$$%
\mbox{(QM3)}\qquad d(x,y)=d(y,x)=0\Rightarrow x=y,
$$%
for all $x,y\in X,$ then $d$ is called a {\it quasi-metric}. The pair $(X,d)$ is called a {\it quasi-pseudometric space}, respectively  a {\it quasi-metric space}\footnote{In \cite{Cobzas} the term ``quasi-semimetric" is used instead of ``quasi-pseudometric", while in \cite{Larrecq} it is called hemi-metric.} The conjugate of the quasi-pseudometric
$d$ is the quasi-pseudometric $\bar d(x,y)=d(y,x),\, x,y\in X.$ The mapping $
d^s(x,y)=\max\{d(x,y),\bar d(x,y)\},\,$ $ x,y\in X,$ is a pseudometric on $X$ which is a metric if and
only if $d$ is a quasi-metric.

If $(X,d)$ is a quasi-pseudometric space, then for $x\in X$ and $r>0$ we define the balls in $X$ by the formulae %
\begin{align*}%
B_d(x,r)=&\{y\in X : d(x,y)<r\} \; \mbox{-\; the open ball, and }\\ %
B_d[x,r]=&\{y\in X : d(x,y)\leq r\} \; \mbox{-\; the closed ball. } %
\end{align*} %

\textbf{Topological properties}

 The topology $\tau_d$ (or $\tau(d)$) of a quasi-pseudometric space $(X,d)$ can be defined starting from the family
$\rond{V}_d(x)$ of neighborhoods  of an arbitrary  point $x\in X$:%
\bequs
\begin{aligned}
V\in \rond{V}_d(x)\;&\iff \; \exists r>0\;\mbox{such that}\; B_d(x,r)\subseteq V\\
                             &\iff \; \exists r'>0\;\mbox{such that}\; B_d[x,r']\subseteq V. %
\end{aligned} %
\eequs

The convergence of a sequence $(x_n)$ to $x$ with respect to $\tau_d,$ called $d$-convergence and
denoted by
$x_n\xrightarrow{d}x,$ can be characterized in the following way %
\bequ\label{char-rho-conv1} %
         x_n\xrightarrow{d}x\;\iff\; d(x,x_n)\to 0. %
\eequ %

Also
\bequ\label{char-rho-conv2} %
         x_n\xrightarrow{\bar d}x\;\iff\;\bar d(x,x_n)\to 0\; \iff\; d(x_n,x)\to 0\,,
\eequ
and
\bequ\label{char-rho-conv3} \begin{aligned}
 x_n\xrightarrow{d^s}x&\iff\; d^s(x,x_n)\to 0\\&\iff  d(x,x_n)\to 0 \;\mbx{ and }\; d(x_n,x)\to 0\\ &\iff x_n\xrightarrow{d}x\;\mbx{ and }\;  x_n\xrightarrow{\bar d}x\; ,
\end{aligned}\eequ

As a space equipped with two topologies, $\tau_d$ and   $\tau_{\bar d}\,$, a quasi-pseudometric space can be viewed as a bitopological space in the sense of Kelly \cite{kelly63}. In fact, this is the main example of such a space considered in \cite{kelly63} and, later on, the quasi-uniform spaces were considered as well.

The following   topological properties are true for  quasi-pseudometric spaces.
    \begin{prop}[see \cite{Cobzas}]\label{p.top-qsm1}
   If $(X,d)$ is a quasi-pseudometric space, then the following hold.
   \begin{enumerate}
   \item[\rm 1.] The ball $B_d(x,r)$ is $\tau_d$-open and  the ball $B_d[x,r]$ is
       $\tau_{\bar{d}}$-closed. The ball    $B_d[x,r]$ need not be $\tau_d$-closed.
     \item[\rm 2.]
   The topology $\tau_d$ is $T_0$ if and only if  $d $ is a quasi-metric.   \\ The topology $\tau_d$ is $T_1$ if and only if
   $d(x,y)>0$ for all  $x\neq y$  in $X$.
      \item [\rm 3.]  For every fixed $x\in X,$ the mapping $d(x,\cdot):X\to (\Real,|\cdot|)$ is
   $\tau_d$-usc and $\tau_{\bar d}$-lsc. \\
   For every fixed $y\in X,$ the mapping $d(\cdot,y):X\to (\Real,|\cdot|)$ is $\tau_d$-lsc and
   $\tau_{\bar d}$-usc.

  \end{enumerate}%
     \end{prop} %

The following remarks show that imposing too many conditions on a quasi-pseudometric space it becomes  pseudometrizable.

\begin{remark}[\cite{kelly63}] Let $(X,d)$ be  a quasi-metric space. Then
\ben
  \item[\rm (a)]\  if  the mapping $d(x,\cdot):X\to (\Real,|\cdot|)$ is $\tau_d$-continuous for   every $x\in X,$ then the topology $\tau_d$ is regular;
   \item[\rm (b)] if $\tau_d\subseteq \tau_{\bar{d}}$, then the topology $\tau_{\bar{d}}$ is pseudometrizable;
 \item[\rm (c)] if $d(x,\cdot):X\to (\Real,|\cdot|)$ is $\tau_{\bar{d}}$-continuous for every $x\in X,$ then the topology $\tau_{\bar{d}}$ is pseudometrizable.
\een\end{remark}

\begin{remark}
  The characterization of Hausdorff property (or $T_2$) of quasi-metric spaces can be given in terms of uniqueness of the limits, as in the metric case. The topology of a quasi-pseudometric space $(X,d)$ is Hasudorff if and only if every sequence in $X$ has at most one $d$-limit if and only if every sequence in $X$ has at most one $\bar d$-limit (see \cite{wilson31}).

   In the case of an asymmetric normed space   there exists a characterization in terms of the quasi-norm (see \cite{Cobzas}, Proposition 1.1.40).
  \end{remark}

  Recall that a topological space $(X,\tau)$  is called:

 \begin{itemize}
 \item  $T_0$ if for every pair of distinct points  in $X$, at least one of them has a neighborhood   not containing the other;
   \item  $T_1$ if   for every pair of distinct points in $X$,  each of them has a neighborhood   not containing the other;
\item    $T_2$ (or {\it Hausdorff}) if  every  two distinct points  in $X$ admit  disjoint  neighborhoods;
\item  {\it regular}  if for every   point $x\in X$ and closed set $A$ not containing $x$ there exist the disjoint open sets $U,V$ such that $x\in U$  and $A\subseteq V.$  \end{itemize}

  \textbf{Completeness in quasi-pseudometric spaces}

   The lack of symmetry in the definition of quasi-metric spaces causes a lot of troubles, mainly concerning
   completeness, compactness and total boundedness in such spaces.
   There are a lot of completeness notions in quasi-metric spaces, all agreeing with the usual notion of
   completeness in the metric case,  each of
   them having its advantages and weaknesses (see \cite{reily-subram82}, or \cite{Cobzas}).

   As in what follows we shall work only with two of these  notions, we shall present only them,
   referring to \cite{Cobzas} for   others.

   We use the  notation
   \begin{align*}
  &\Nat=\{1,2,\dots\}  \mbx{ --  the set of natural numbers,}\\
&\Nat_0=\Nat\cup\{0\} \mbx{ --  the set of non-negative integers.}
  \end{align*}

  \begin{defi}\label{def.rKC}
    Let $(X,d)$ be a quasi-pseudometric space. A sequence $(x_n)$ in $(X,d)$ is called:
\bite
\item\;  {\it left $d$-$K$-Cauchy} if  for every $\epsic >0$
   there exists $n_\epsic\in \Nat$ such that
\bequ\label{def.l-Cauchy}\begin{aligned} %
 &\forall n,m, \;\;\mbox{with}\;\; n_\epsic\leq n < m ,\quad d(x_n,x_m)<\epsic\\ %
 \iff &\forall n \geq n_\epsic,\; \forall k\in\Nat,\quad d(x_n,x_{n+k})<\epsic; %
 \end{aligned}\eequ
\item\;   {\it right $d$-$K$-Cauchy} if  for every $\epsic >0$
   there exists $n_\epsic\in \Nat$ such that
\bequ\label{def.r-Cauchy}\begin{aligned} %
 &\forall n,m, \;\;\mbox{with}\;\; n_\epsic\leq n < m ,\quad d(x_m,x_n)<\epsic\\ %
 \iff &\forall n \geq n_\epsic,\; \forall k\in\Nat,\quad d(x_{n+k},x_n)<\epsic. %
 \end{aligned}\eequ\eite
  \end{defi}

The quasi-pseudometric space $(X,d)$ is called:
\bite
\item\;  {\it sequentially left $d$-$K$-complete} if every left $d$-$K$-Cauchy sequence is
$d$-convergent;
\item\;  {\it sequentially  right $d$-$K$-complete} if every right  $d$-$K$-Cauchy sequence is
$d$-convergent;
\item\;  {\it sequentially  left (right) Smyth complete} if every left (right)  $d$-$K$-Cauchy sequence is
$d^s$-convergent.\eite

\begin{remark}\hfill\begin{itemize}
\item[{\rm 1.}]
It is obvious that a sequence is left $d$-$K$-Cauchy
   if and only if it is right $\bar{d}$-$K$-Cauchy.  Also a left (right) Smyth complete quasi-pseudometric space is left (right) $K$-complete
    and the space $(X,d)$ is right Smyth complete
 if and only if $(X,\bar d)$ is left Smyth complete. For this reason, some authors call a Smyth complete  space a left Smyth complete.
\item[{\rm 2.}]  The notion of Smyth completeness, introduced by Smyth in \cite{smyth88} (see also \cite{smyth94}),  is an important notion in quasi-metric and quasi-uniform spaces as well as for the applications to theoretical computer science (see, for instance, \cite{romag-val08}, \cite{shellek95}). A good presentation of this notion is given in Section 7.1 of the book \cite{Larrecq}.
\item[{\rm 3.}] There are examples showing that a $d$-convergent sequence need not  be
left $d$-$K$-Cauchy, showing that in the asymmetric case the situation is far more complicated than in the
symmetric one (see \cite{reily-subram82}).
\item[{\rm 4.}]
  If each convergent sequence in a regular quasi-metric space $(X,d)$
   admits a left $K$-Cauchy subsequence, then $X$ is metrizable
   (\cite{kunzi-reily93}).
   \end{itemize} \end{remark}

   \begin{remark}\hfill\begin{itemize}
       \item[\rm 1.] One can define  more general notions of completeness by replacing in Definition \ref{def.rKC} the sequences with nets. Stoltenberg \cite[Example 2.4]{stolt69} gave an example of a sequentially right $K$-complete $T_1$ quasi-metric space which is not right $K$-complete (i.e., not right $K$-complete by nets). See \cite{cobz20} for some further specifications.
 \item[\rm 2.]  In the case of Smyth completeness, the completeness by nets is equivalent to the completeness by sequences (see \cite{romag15}). Also, the left (or right) Smyth completeness implies
 the completeness of the pseudometric space $(X,d^s)$. In this case one says that the quasi-pseudometric space $(X,d)$ is bicomplete.
\end{itemize}\end{remark}

The following result is the quasi-pseudometric analog of a well-known property in metric spaces.

\begin{prop}[see \cite{Cobzas}, Section 1.2]\label{p1.rKC} Let $(X,d)$ be a quasi-pseudometric space. If a right $K$-Cauchy sequence $ (x_n)$ contains a subsequence $d$-convergent ($\bar d$-convergent, $d^s$-convergent) to some $x\in X$, then the sequence $(x_n)$ is $d$-convergent ($\bar d$-convergent, $d^s$-convergent) to $x$.
\end{prop}

\subsection{Ekeland principle in quasi-pseudometric spaces}
 The following  version of Ekeland variational principle in quasi-pseudometric spaces was proved in \cite{cobz19}. For  a quasi-pseudometric space $X$,  a function $f:X\to\Real\cup\{\infty\},\,\alpha >0$ and $x\in X$ put
\bequ\label{def.Sx-Jx}
S_\alpha(x)=\{y\in X : f(y)+\alpha d(y,x)\le f(x)\}\,.\eequ

\begin{theo}\label{t.Ek2-qm} Let $(X,d)$ be a sequentially right $K$-complete quasi-pseudometric space and $f:X\to\Real\cup\{\infty\}$ a proper bounded below  lsc function.
Given  $\epsic,\lbda >0$ and  $x_0\in \dom f$ there exists $z\in X$ such that
\bequ\label{eq1.Ek2-qm}\begin{aligned}
{\rm (i)}\quad&  f(z)+\frac{\epsic}{\lbda}d(z,x_0)\le f(x_0);\\
{\rm (ii)}\quad& f(y)=f(z) \;\mbx{ for all  }\; y\in S_\gamma(z);\\
{\rm (iii)}\quad& f(z)<f(x)+\frac{\epsic}{\lbda}d(x,z)\quad\mbx{for all }\; x\in X\sms S_\gamma(z)\,,
\end{aligned}\eequ
where $\gamma=\eps/\lbda$.

If, further,
$
f(x_0)\le\epsic+\inf f(X),
$
then
\bequs\label{eq3.Ek2-qm}
{\rm (iv)}\quad
d(z,x_0)\le\lbda.\qquad\qquad\qquad\qquad\qquad\qquad\qquad\qquad\;
\eequs\end{theo}

Obviously, an analog of Theorem \ref{t.EkVPb} holds in this case too.
\begin{theo}\label{t.Ek2b-qm} Let $(X,d)$ be a sequentially right $K$-complete quasi-pseudometric space and $f:X\to\Real\cup\{\infty\}$ a proper bounded below  lsc function.
Given  $\epsic,\lbda' >0$ and  $x_0\in \dom f$ there exists $z\in X$ such that
\bequ\label{eq1.Ek2-qm}\begin{aligned}
{\rm (i')}\quad&  f(z)+\ \lbda'd(z,x_0)\le f(x_0);\\
{\rm (ii')}\quad& f(y)=f(z) \;\mbx{ for all  }\; y\in S_{\lbda'}(z);\\
{\rm (iii')}\quad& f(z)<f(x)+ \lbda'd(x,z)\quad\mbx{for all }\; x\in X\sms S_{\lbda'}(z)\,.
\end{aligned}\eequ

If, further,
$
f(x_0)\le\epsic+\inf f(X),
$
then
\bequs\label{eq3.Ek2-qm}
{\rm (iv')} \quad
d(z,x_0)\le\eps/\lbda'.\qquad\qquad\qquad\qquad\qquad\qquad \qquad\;\,
\eequs\end{theo}

The proof of Theorem \ref{t.Ek2-qm} is based on the properties of Picard sequences corresponding to the set-valued map  $S_\alpha:X\rightrightarrows X$. A sequence $(x_n)_{n=0}^\infty$ in $X$ is called a  \emph{Picard sequence}  for  $S_\alpha$ if $x_{n+1}\in S_\alpha(x_n)$ for all $n\in\Nat_0,$  for a given $x_0\in X$. We mention some of the properties of these sets $S_\alpha(x)$ which will be used in what follows.

Let $(X,d)$ be a quasi-pseudometric space and $f:X\to\Real\cup\{\infty\}$ a proper function, i.e., $$\dom f:=\{x\in X: f(x)<\infty\}\ne\ety.$$ It is obvious that $S_\alpha(x)=X$ if $f(x)=\infty$ and
$$
S_\beta(x)\sse S_\alpha(x)$$
for $0<\alpha<\beta.$

\begin{prop}\label{p1.Sx} Let $(X,d)$ be a quasi-pseudometric  space, $f:X\to\Real\cup\{\infty\}$ a proper function, $\alpha>0$ and $x\in\dom f.$
 The set $S_\alpha(x)$ has the following properties:
\bequ\label{eq2.Sx}
\begin{aligned}
  {\rm (i)}\;\; &x\in S_\alpha(x)\quad\mbx{and}\quad S_\alpha(x)\sse\dom f;\\
  {\rm (ii)}\;\; &y\in S_\alpha(x) \Ra f(y)\le f(x)\;\mbx{ and }\; S_\alpha(y)\subseteq S_\alpha(x);\\
  {\rm (iii)}\;\; &y\in S_\alpha(x)\sms\ov{\{x\}} \Ra f(y)<f(x);\\
  {\rm (iv)}\;\; &\mbx{if $f$ is bounded below, then}\\&S_\alpha(x)\sms \ov{\{x\}}\ne\ety \Ra f(x)>\inf f(S_\alpha(x));\\
  {\rm (v)}\;\; &\mbx{if $f$ is lsc, then }  S_\alpha(x)\mbx{ is closed}.
  \end{aligned}\eequ
\end{prop}

The key result used in the proofs of various variational principles in \cite{cobz19} is the following.
\begin{prop}[\cite{cobz19}, Prop. 2.14]\label{p1.Pic}
If the space $ (X,d) $ is sequentially right $K$-complete and the function $f$ is bounded below and  lsc,  then there exists a point $ z\in X $     such that
 \bequ  \label{eq3.Pic}\begin{aligned}
      {\rm (i)}&\quad f(y)=f(z)=\inf f(S_\alpha(z))\;\mbx{ and} \\
   {\rm (ii)}&\quad S_\alpha(y)\sse\ov{\{y\}}\,,
   \end{aligned}\eequ
 for all $ y\in S_\alpha(z).$
 \end{prop}

 \begin{remark}
In fact, in \cite{cobz19}, Proposition \ref{p1.Pic} is proved in a slightly more general context, namely for a nearly lsc function $f$, meaning  that
$$
f(x)\le \liminf_{n\to\infty}f(x_n)\,,$$
for every sequence $(x_n)$ in $X$ with pairwise distinct terms in $X$ such that $x_n\xrightarrow{d}x.$
 \end{remark}

\subsection{The strong Ekeland principle -- Georgiev's version}

   We show that Turinici  proof  \cite{turin05} of the strong EkVP (Theorem \ref{t.str-EkVP1}) can be adapted to obtain a proof of a quasi-pseudometric version  of the strong Ekeland Variational Principle.

  \begin{theo}\label{t.sEk-qm} Let $(X,d)$ be a sequentially right $K$-complete quasi-pseudometric space and $f:X\to\Real\cup\{\infty\}$ a proper bounded below  lsc function.
Given  $\gamma,\delta >0$ and  $x_0\in \dom f$ there exists $z\in X$ such that
\bequ\label{eq1.sEk-qm}\begin{aligned}
{\rm (a)}\quad&  f(z)+\gamma d(z,x_0)\le f(x_0)+\delta;\\
{\rm (b)}\quad& f(y)=f(z) \;\mbx{ for all  }\; y\in S_{\gamma}(z);\\
{\rm (c)}\quad& f(z)<f(x)+ \gamma d(x,z)\quad\mbx{for all }\; x\in X\sms S_{\gamma}(z);\\
{\rm (d)}\quad&  f(x_n)+\gamma d(z,x_n)\to f(z)\;\Ra\; d(x_n,z)\to 0,\\&\mbox{for every  sequence } (x_n) \mbx{  in }  X.
\end{aligned}\eequ
 \end{theo}\begin{proof} Let
 $$
 X_0=\{y\in X : f(x)\le f(x_0)+\delta\}\,.$$

 Then $x_0\in X_0$ and $X_0$ is closed (because $f$ is lsc)  and so sequentially right $K$-complete. Also
 \bequ\label{eq.sEk-qm-inf}
 \inf f(X_0)=\inf f(X)\,.
 \eequ

 Indeed, if  $m:=\inf f(X)$ and $M:=\inf f(X_0)$, then $m\le M.$ Let $(x_n)$ be a sequence in $X$ such that $f(x_n)\to m$ as $n\to\infty.$ Then there exists $n_0\in \Nat$ such that
$f(x_n)\le m+\delta\le f(x_0)+\delta$,  that is, $x_n\in X_0,$ for all $n\ge n_0$. But then $M\le f(x_n),\,\forall n\ge n_0,$ which  for $n\to \infty$ yields  $ M\le m,$ and so  $m=M.$

 Let $0<\lbda<1$ be such that
  \bequ\label{eq2.sEk-qm}
  \frac{\lbda}{1-\lbda}\left(f(x_0)-\inf f(X)\right)\le\delta\,.
  \eequ

  By  Theorem \ref{t.Ek2b-qm} applied  to $ X_0,\, f|_{X_0} $ and $ \lbda':=(1-\lbda)\gamma $,   there exists $z\in X_0$ such that
 \bequ\label{eq3.sEk-qm}
 \begin{aligned}
{\rm (i)}\quad&  f(z)+\ \lbda'd(z,x_0)\le f(x_0);\\
{\rm (ii)}\quad& f(y)=f(z) \;\mbx{ for all  }\; y\in X_0\cap S_{\lbda'}(z)=S_{\lbda'}(z);\\
{\rm (iii)}\quad& f(z)<f(x)+ \lbda'd(x,z)\quad\mbx{for all }\; x\in X_0\sms S_{\lbda'}(z)\,.
\end{aligned}\eequ

To justify the equality $X_0\cap S_{\lbda'}(z)=S_{\lbda'}(z)$ in (ii) above, observe that
$$
z\in X_0\;\Ra\; S_{\lbda'}(z)\sse X_0\,.$$

Indeed, the existence of an element $x\in (X\sms X_0)\cap S_{\lbda'}(z)$ would yield the contradiction:
$$
f(x_0)+\delta <f(x)\le f(x)+\lbda' d(x,z)\le f(z)\le f(x_0)+\delta\,.$$

 By \eqref{eq3.sEk-qm}.(i), the definition of $\lbda'$ and  \eqref{eq2.sEk-qm},
 \begin{align*}
 \gamma d(x_0,z)&\le \frac1{1-\lambda}\left[f(x_0)-f(z)\right]\\
 &=f(x_0)-f(z)+\frac{\lambda}{1-\lambda}\left[f(x_0)-f(z)\right]\\
 &\le f(x_0)-f(z)+\frac{\lambda}{1-\lambda}\left[f(x_0)-\inf f(X)\right]\\
 &\le f(x_0)-f(z)+\delta\,,
   \end{align*}
showing that condition \eqref{eq1.sEk-qm}.(a) holds.

The inequality $\lbda'=(1-\lbda)\gamma<\gamma$ implies
$$
S_\gamma(z)\sse S_{\lbda'}(z)\,,$$
so that, by \eqref{eq3.sEk-qm}.(ii),
$ f(y)=f(z)$
for all $y\in S_\gamma(z),$ i.e., \eqref{eq1.sEk-qm}.(b) holds too.

The inequality \eqref{eq1.sEk-qm}.(c) follows from the definition of the set $S_\gamma(z).$

 Observe now that, by the definition of the set $S_{\lbda'}(z)$,
\bequ\label{eq4.sEk-qm}
f(z)<f(x)+(1-\lbda)\gamma d(x,z)\quad\mbx{for all}\quad x\in X\sms S_{\lbda'}(z)\,.
\eequ

 To  prove \eqref{eq1.sEk-qm}.(d), let
 $(x_n) $ be a sequence in $X$ such that $$\lim_{n\to\infty}\left[f(x_n)+\gamma d(z,x_n)\right]=f(z)\,. $$

 If $x_n\in S_{\lbda'}(z),$ then, by \eqref{eq3.sEk-qm}.(ii) $f(x_n)=f(z)$ and the inequality    $f(x_n)+\lbda' d(x_n,z)\le  f(z)$ implies $d(x_n,z)=0.$

 For all $n$ such that $x_n\in X\sms S_{\lbda'}(z)$ the inequality \eqref{eq4.sEk-qm} yields
 $$
 \lbda\gamma d(x_n,z)< f(x_n)+\gamma d(x_n,z)-f(z)\,\lra 0\;\mbx{ as }\; n\to \infty\,.$$

 Consequently,
 $$\lim_{n\to\infty}d(x_n,z)=0\,.$$
  \end{proof}

  \begin{remark} Actually, condition \eqref{eq1.sEk-qm}.(d)  says that the minimizing sequence $(x_n)$ is $\bar d$-convergent to $z$.
  \end{remark}

  \subsection{The strong Ekeland principle -- Suzuki's versions}

  As we have seen in Subsection \ref{Ss.qpm} completeness in quasi-pseudometric spaces has totally different features than that in  metric spaces. The situation is the same with compactness,
  see \cite{Cobzas}.

  In order to extend Theorem \ref{T1.str-EkVP-comp} to quasi-pseudometric spaces we consider the following notion.  A subset $Y$ of a quasi-pseudometric space $(X,d)$ is called
  $d$-\emph{bounded} if there exist $x\in X$ and $r>0$ such that
  \beqs
  Y\sse B_d[x,r]\,,\eeqs
or, equivalently,
  \beqs
  \sup\{d(x,y) :y\in Y\}<\infty\quad\mbx{for every}\quad x\in X.
  \eeqs

  We say that a  sequence $(x_n)_{n\in\Nat}$ in $X$ is  $d$-bounded if the set $\{x_n:n\in\Nat\}$ is $d$-bounded.

  Similar definitions are given for $\bar d$-boundedness.

  We have seen (Remark \ref{re.bd-comp}) that a boundedly compact metric space is complete.
  In the case of quasi-pseudometric spaces we have.
  \begin{prop}\label{p1.bd-Smyth}Let $(X,d)$ be a quasi-pseudometric space.
  If every $\bar d$-bounded sequence in $X$ contains a $d^s$-convergent subsequence, then the space $X$ is right Smyth complete.
  \end{prop}\begin{proof}
    Let  $(x_n)$ be a right $K$-Cauchy sequence in $X$. Then $(x_n)$ is $\bar d$-bounded. Indeed, for $\eps =1$ there exists $n_1\in\Nat$ such that
    $$
    d(x_n,x_{n_1})\le 1\quad\mbx{fro all}\quad n\ge n_1\,,$$
    which implies the $\bar d$-boundedness of $(x_n)$. It follows that $(x_n)$ contains a subsequence $d^s$-convergent to some $x\in X.$  By Proposition \ref{p1.rKC}
    the sequence $(x_n)$ is $d^s$-convergent to  $x$.
  \end{proof}

  The analogs of the conditions \eqref{eq1.str-EkVP-Suz} in the quasi-pseudometric case are:
  \bequ\label{eq1.sEk-Suz-qpm}\begin{aligned}
  &{\rm (i)}\;\; f(z)+\lbda \rho(z,x_0)\le f(x_0)\,;\\
   &{\rm (ii)}\;\; f(y)=f(z) \quad\mbox{for all}\quad y\in S_\lbda(z);\\
  &{\rm (iii)}\;\; f(z)<f(x)+\lbda \rho(x,z)\quad\mbox{for all}\quad x\in X\setminus S_\lbda(z)\,;\\
  &{\rm (iv)}\;\;  f(x_n)+\lbda \rho(z,x_n)\to f(z)\;\Ra\; \lim_{n\to\infty} d(x_n,z)=0\;\;\\
  &\qquad\; \mbox{for every  sequence } (x_n)\mbx{  in }  X.
  \end{aligned}\eequ

  The quasi-pseudometric analog of Theorem \ref{T1.str-EkVP-comp} is the following.
  \begin{theo}\label{T1.str-Ek-Suz} Let $(X,d)$  be a quasi-pseudometric space such that every $\bar d$-bounded sequence in $X$ contains a $d^s$-convergent subsequence and $f:X\to\Real\cup\{\infty\}$ a proper bounded below $d$-lsc function. Then for every $x_0\in X$ and $\lbda >0$ there exists a point $z\in X$ satisfying \eqref{eq1.sEk-Suz-qpm}.
  \end{theo}\begin{proof}
    By Theorem \ref{t.Ek2b-qm} there exists $z\in X$ such that
    \bequ\label{eq1.Ek2-qm-Suz}\begin{aligned}
{\rm (a)}\quad&  f(z)+\ \lbda d(z,x_0)\le f(x_0);\\
{\rm (b)}\quad& f(y)=f(z) \;\mbx{ for all  }\; y\in S_{\lbda }(z);\\
{\rm (c)}\quad& f(z)<f(x)+ \lbda d(x,z)\quad\mbx{for all }\; x\in X\sms S_{\lbda}(z)\,.
\end{aligned}\eequ

Let $(x_n)$ be a sequence in $X$ such that
\bequ\label{eq2.sEk2-qm-Suz}
\lim_{n\to\infty}[f(x_n)+ \lbda d(x_n,z)]=f(z)\,.
\eequ

Suppose that $(d(x_n,z))_{n\in\Nat}$ does not converge to 0. Then there exist $\gamma >0$ and a subsequence $(x_{n_k})_{k\in\Nat} $ of $(x_n)$ such that
$d(x_{n_k},z)\ge \gamma$ for all $k\in\Nat.$ Passing to this sequence we can suppose, without restricting the generality, that the sequence $(x_n)$ satisfies
\eqref{eq2.sEk2-qm-Suz}  and  that
\bequ\label{eq3.sEk2-qm-Suz}
d(x_n,z)\ge \gamma\,,
\eequ
for all $n\in\Nat.$

Let  $n_1\in\Nat$ be such that
$$
f(x_n)+ \lbda d(x_n,z)\le f(z)+1\,
$$
for all $n\ge n_1.$ Then
\begin{align*}
  \lbda d(x_n,z)&=f(x_n)+ \lbda d(x_n,z)-f(x_n)\\
  &\le f(z)+1-\inf f(X)\,,
\end{align*}
for all $n\ge n_1,$ which shows that the sequence $(x_n)$ is $\bar d$-bounded.  By hypothesis, it contains a subsequence $(x_{n_k})_{k\in\Nat} $ $\,d^s$-convergent to some $y\in X$.

Observe that \bequ\label{eq4.sEk2-qm-Suz}y\notin\ov{\{z\}}^d.\eequ
 Indeed,
$$
y\in\ov{\{z\}}^d\iff d(y,z)=0\,,$$
which would imply
$$
d(x_{n_k},z)\le  d(x_{n_k},y)+d(y,z)=  d(x_{n_k},y)\to 0\;\mbx{ as }\; k\to\infty\,,$$
in contradiction  to \eqref{eq3.sEk2-qm-Suz}.

Since $S_\lbda(z)\sse \ov{\{z\}}^d$  (see Proposition \ref{p1.Pic}), \eqref{eq4.sEk2-qm-Suz} implies $y\notin  S_\lbda(z).$ Taking into account  \eqref{eq1.Ek2-qm-Suz}.(c) and the $d$-lsc of $f$ and $d(\cdot,z)$, one obtains the contradiction
$$
f(z)<f(y)+\lbda d(y,z)\le \lim_{n\to\infty}[f(x_n)+ \lbda d(x_n,z)]=f(z)\,.$$

Consequently, we must have $\lim_{n\to\infty}d(x_n,z)=0.$
  \end{proof}

  \section{Conclusions}
  We have proved (Theorem \ref{T1.str-Ek-Suz}) a version of strong Ekeland in a quasi-pseudometric space $(X,d)$  having $d^s$-compact $\bar d$-bounded sets.
  As it was shown by Suzuki \cite{suzuki10}, in a metric space $X$  the validity of strong EkVP is equivalent to the fact all closed bounded subsets of $X$ are compact.
  I do not know whether a similar result holds in the asymmetric case - a  question that deserves further investigation.

  A notion of reflexivity of normed spaces was also considered in the asymmetric case (see \cite{raffi-rom-per03a} or \cite[Section 2.5.6]{Cobzas}), but  in a more complicated way than in the classical one. The extension of Theorems \ref{T1.str-EkVP-refl} and  \ref{t.str-EkVP-refl} to the asymmetric case could be another theme of reflection.

\end{document}